\documentclass[12pt]{amsart} 
\usepackage{amsfonts,graphics,amsmath,amsthm,amsfonts,amscd,amssymb,amsmath,latexsym,multicol}
\usepackage{epsfig,url}
\usepackage{flafter}
\usepackage{hyperref}

\usepackage[small,nohug,UglyObsolete]{diagrams}
\diagramstyle[labelstyle=\scriptstyle]

\makeatletter

\def\jobis#1{FF\fi
  \def\preedicate{#1}%
  \edef\preedicate{\expandafter\strip@prefix\meaning\preedicate}%
  \edef\job{\jobname}%
  \ifx\job\preedicate
}

\makeatother

\if\jobis{proposal}%
 \def\try{subsection}%
\else
  \def\try{section}%
\fi

\theoremstyle{plain}
\newtheorem{theorem}{Theorem}[\try]
\newtheorem{corollary}[theorem]{Corollary}
\newtheorem{lemma}[theorem]{Lemma}

\newtheorem{definition-lemma}[theorem]{Definition-Lemma}
\newtheorem{question}[theorem]{Question}
\newtheorem{definition}[theorem]{Definition}

\newtheorem{conjecture}[theorem]{Conjecture}

\def\ideal#1.{I_{#1}}
\def\ring#1.{\mathcal {O}_{#1}}
\def\fring#1.{\hat{\mathcal {O}}_{#1}}
\def\proj#1.{\mathbb {P}(#1)}
\def\pr #1.{\mathbb {P}^{#1}}
\def\dpr #1.{\hat{\mathbb {P}}^{#1}}
\def\af #1.{\mathbb A^{#1}}
\def\Hz #1.{\mathbb F_{#1}}
\def\Hbz #1.{\overline{\mathbb F}_{#1}}
\def\fb#1.{\underset #1 {\times}}
\def\rest#1.{\underset {\ \ring #1.} \to \otimes}
\def\au#1.{\operatorname {Aut}\,(#1)}
\def\deg#1.{\operatorname {deg } (#1)}
\def\pic#1.{\operatorname {Pic}\,(#1)}
\def\pico#1.{\operatorname{Pic}^0(#1)}
\def\picg#1.{\operatorname {Pic}^G(#1)}
\def\ner#1.{NS (#1)}
\def\rdown#1.{\llcorner#1\lrcorner}
\def\rfdown#1.{\lfloor{#1}\rfloor}
\def\rup#1.{\ulcorner{#1}\urcorner}
\def\rcup#1.{\lceil{#1}\rceil}
\def\cone#1.{\operatorname {NE}(#1)}
\def\ccone#1.{\overline{\operatorname {NE}}(#1)}
\def\coef#1.{\frac{(#1-1)}{#1}}
\def\vit#1.{D_{\langle #1 \rangle}}
\def\mm#1.{\overline {M}_{0,#1}}
\def\H1#1.{H^1(#1,{\ring #1.})}
\def\ac#1.{\overline {\mathbb F}_{#1}}

\def\adj#1.{\frac {#1-1}{#1}}
\def\spn#1.{\overline{#1}}
\def\pek#1.#2.{\Cal P^{#1}(#2)}
\def\plk#1.#2.{\Cal P^{\leq #1}(#2)}
\def\ev#1.{\operatorname{ev_{#1}}}
\def\ilist#1.{{#1}_1,{#1}_2,\dots}
\def\bminv#1.{(\nu_1,s_1;\nu_2,s_2;\dots ;\nu_{#1},s_{#1};\nu_{r+1})}
\def\zinv#1.{(\nu_1,s_1;\nu_2,s_2;\dots ;\nu_{#1},s_{#1};0)}
\def\iinv#1.{(\nu_1,s_1;\nu_2,s_2;\dots ;\nu_{#1},s_{#1};\infty)}

\def\llist#1.#2.{{#1}_1,{#1}_2,\dots,{#1}_{#2}}
\def\ulist#1.#2.{{#1}^1,{#1}^2,\dots,{#1}^{#2}}
\def\lomitlist#1.#2.{{#1}_1,{#1}_2,\dots,\hat {{#1}_i}, \dots, {#1}_{#2}}
\def\lomitlistz#1.#2.{{#1}_0,{#1}_1,\dots,\hat {{#1}_i}, \dots, {#1}_{#2}}
\def\loc#1.#2.{\Cal O_{#1,#2}}
\def\fderiv#1.#2.{\frac {\partial #1}{\partial #2}}
\def\deriv#1.#2.{\frac {d #1}{d #2}}
\def\map#1.#2.{#1 \longrightarrow #2}
\def\rmap#1.#2.{#1 \dasharrow #2}
\def\emb#1.#2.{#1 \hookrightarrow #2}
\def\non#1.#2.{\text {Spec }#1[\epsilon]/(\epsilon)^{#2}}
\def\Hi#1.#2.{\text {Hilb}^{#1}(#2)}
\def\sym#1.#2.{\operatorname {Sym}^{#1}(#2)}
\def\Hb#1.#2.{\text {Hilb}_{#1}(#2)}
\def\Hm#1.#2.{\Hom_{#1}(#2)}
\def\prd#1.#2.{{#1}_1\cdot {#1}_2\cdots {#1}_{#2}}
\def\Bl #1.#2.{\operatorname {Bl}_{#1}#2}
\def\pl #1.#2.{#1^{\otimes #2}}
\def\mgn#1.#2.{\overline {M}_{#1,#2}}
\def\ialist#1.#2.{{#1}_1 #2 {#1}_2, #2\dots}
\def\pair#1.#2.{\langle #1, #2\rangle}
\def\vandermonde#1.#2.{\left|
\begin{matrix}
1 & 1 & 1 & \dots & 1\\
{#1}_1 & {#1}_2 & {#1}_3 & \dots & {#1}_{#2}\\
{#1}_1^2 & {#1}_2^2 & {#1}_3^2 & \dots & {#1}_{#2}^2\\
\vdots & \vdots & \vdots & \ddots & \vdots\\
{#1}_1^{#2-1} & {#1}_2^{#2-1} & {#1}_2^{#2-1} & \dots & {#1}_{#2}^{#2-1}\\
\end{matrix}
\right|
}
\def\vandermondet#1.#2.{\left|
\begin{matrix}
1 & {#1}_1   & {#1}_1^2 & \dots & {#1}_1^{#2-1}\\
1 & {#1}_2   & {#1}_2^2 & \dots & {#1}_2^{#2-1}\\
1 & {#1}_3   & {#1}_3^2 & \dots & {#1}_3^{#2-1}\\
\vdots & \vdots & \vdots & \ddots & \vdots\\
1 & {#1}_{#2}& {#1}_{#2}^2 & \dots & {#1}_{#2}^{#2-1}\\
\end{matrix}
\right|
}
\def\gr#1.#2.{\mathbb{G}(#1,#2)}

\def\alist#1.#2.#3.{{#1}_1 #2 {#1}_2 #2\dots #2 {#1}_{#3}}
\def\zlist#1.#2.#3.{#1_0 #2 #1_1 #2\dots #2 #1_{#3}}
\def\lomitlist30#1.#2.#3.{{#1}_0,{#1}_1 #2 \dots #2\hat {{#1}_i} #2\dots #2 {#1}_{#3}}
\def\lmap#1.#2.#3.{#1 \overset{#2}{\longrightarrow} #3}
\def\mes#1.#2.#3.{#1 \longrightarrow #2 \longrightarrow #3}
\def\ses#1.#2.#3.{0\longrightarrow #1 \longrightarrow #2 \longrightarrow #3 \longrightarrow 0}
\def\les#1.#2.#3.{0\longrightarrow #1 \longrightarrow #2 \longrightarrow #3}
\def\res#1.#2.#3.{#1 \longrightarrow #2 \longrightarrow #3\longrightarrow 0}
\def\Hi#1.#2.#3.{\text {Hilb}^{#1}_{#2}(#3)}
\def\ten#1.#2.#3.{#1\underset {#2}{\otimes} #3}
\def\lomitlist30#1.#2.#3.{{#1}_0 #2 {#1}_1 #2 \dots #2 \hat {{#1}_i} #2 \dots #2 {#1}_{#3}}
\def\mderiv#1.#2.#3.{\frac {d^{#3} #1}{d #2^{#3}}}

\def\Hom{\operatorname{Hom}}

\def\dim{\operatorname{dim}}

\def\deg{\operatorname{deg}}

\def\pgl{\operatorname{PGL}}

\def\rest{\operatorname{res}}

\def\WDiv{\operatorname{WDiv}}

\def\e{\Cal E}

\def\e1{E_1}
\def\e2{E_2}

\def\ds{\displaystyle}

\def\mapdown#1{\big\downarrow\rlap{$\vcenter
{\hbox{$\scriptstyle#1$}}$}}

\def\mapse#1{
{\vcenter{\hbox{$\mathop{\smash{\raise1pt\hbox{$\diagdown$}\!\lower7pt
\hbox{$\searrow$}}\vphantom{p}}\limits_{#1}\vphantom{\mapdown{}}$}}}}

\def\VR#1.{height#1pt&\omit&&\omit&&\omit&&\omit&&\omit&\cr}

\def\VRT#1.{height#1pt&\omit&&\omit&\cr}

\usepackage{graphicx}
\begin{document}
\title{The Sarkisov program}
\date{\today}
\author{Christopher D. Hacon} 
\address{Department of Mathematics \\  
University of Utah\\  
155 South 1400 East\\
JWB 233\\
Salt Lake City, UT 84112-0090, USA}
\email{hacon@math.utah.edu}
\author{James M\textsuperscript{c}Kernan} 
\address{Department of Mathematics\\ 
University of California at Santa Barbara\\ 
Santa Barbara, CA 93106, USA} 
\email{mckernan@math.ucsb.edu}
\address{Department of Mathematics\\ 
MIT\\ 
77 Massachusetts Avenue\\
Cambridge, MA 02139, USA}
\email{mckernan@math.mit.edu} 

\thanks{The first author was partially supported by NSF grant no: 0757897 and by the Clay
  Mathematics Institute.  The second author was partially supported by NSA grant no:
  H98230-06-1-0059, NSF grant no: 0701101 and an Eisenbud fellowship.  Some of this work
  was done whilst both authors were visiting MSRI and both authors would like to thank
  MSRI for its hospitality.}

\begin{abstract} Any two birational Mori fibre spaces are connected by a sequence of
Sarkisov links.
\end{abstract}

\maketitle

\tableofcontents

\section{Introduction}

We prove that any two birational Mori fibre spaces are connected by a sequence of
elementary transformations, known as Sarkisov links:
\begin{theorem}\label{t_main} Suppose that $\phi\colon\map X.S.$ and $\psi\colon\map
Y.T.$ are two Mori fibre spaces with $\mathbb{Q}$-factorial terminal singularities.  

Then $X$ and $Y$ are birational if and only if they are related by a sequence of Sarkisov
links.
\end{theorem}

Recall the following:
\begin{conjecture}\label{c_mori} Let $(Z,\Phi)$ be a kawamata log terminal pair.  

Then we may run $f\colon\rmap Z.X.$ the $(K_Z+\Phi)$-MMP such that either
\begin{enumerate} 
\item $(X,\Delta)$ is a log terminal model, that is $K_X+\Delta$ is nef, or 
\item there is a Mori fibre space $\phi\colon\map X.S.$, that is $\rho(X/S)=1$ and
$-(K_X+\Delta)$ is $\phi$-ample,
\end{enumerate} 
where $\Delta=f_*\Phi$.  
\end{conjecture}

We will refer to the log terminal model $X$ and the Mori fibre space $\phi$ as the output
of the $(K_Z+\Phi)$-MMP.  If $h\colon\rmap Z.X.$ is any sequence of divisorial
contractions and flips for the $(K_Z+\Phi)$-MMP then we say that $h$ is the result of
running the $(K_Z+\Phi)$-MMP.  In other words if $h$ is the result of running the
$(K_Z+\Phi)$-MMP then $X$ does not have to be either a log terminal model or a Mori fibre
space.

By \cite{BCHM06} the only unknown case of \eqref{c_mori} is when $K_Z+\Phi$ is
pseudo-effective but neither $\Phi$ nor $K_Z+\Phi$ is big.  Unfortunately the output is
not unique in either case.  We will call two Mori fibre spaces $\phi\colon\map X.S.$ and
$\psi\colon\map Y.T.$ \textit{Sarkisov related} if $X$ and $Y$ are outcomes of running the
$(K_Z+\Phi)$-MMP, for the same $\mathbb{Q}$-factorial kawamata log terminal pair
$(Z,\Phi)$.  This defines a category, which we call the Sarkisov category, whose objects
are Mori fibre spaces and whose morphisms are the induced birational maps $\rmap X.Y.$
between two Sarkisov related Mori fibre spaces.  Our goal is to show that every morphism
in this category is a product of Sarkisov links.  In particular a Sarkisov link should
connect two Sarkisov related Mori fibre spaces.  

\begin{theorem}\label{t_sarkisov} If $\phi\colon\map X.S.$ and $\psi\colon\map Y.T.$ 
are two Sarkisov related Mori fibres spaces then the induced birational map
$\sigma\colon\rmap X.Y.$ is a composition of Sarkisov links.
\end{theorem}

Note that if $X$ and $Y$ are birational and have $\mathbb{Q}$-factorial terminal
singularities, then $\phi$ and $\psi$ are automatically the outcome of running the
$K_Z$-MMP for some projective variety $Z$, so that \eqref{t_main} is an easy 
consequence of \eqref{t_sarkisov}.  

It is proved in \cite{BCHM06} that the number of log terminal models is finite if either $\Phi$
or $K_Z+\Phi$ is big, and it is conjectured that in general the number of log terminal models
is finite up to birational automorphisms.  Moreover Kawamata, see \cite{Kawamata07}, has
proved:
\begin{theorem}\label{t_minimal} Suppose that $\sigma\colon\rmap X.Y.$ is a birational 
map between two $\mathbb{Q}$-factorial varieties which is an isomorphism in codimension
one.  

If $K_X+\Delta$ and $K_Y+\Gamma$ are kawamata log terminal and nef and $\Gamma$ is the
strict transform of $\Delta$ then $\sigma$ is the composition of $(K_X+\Delta)$-flops.
\end{theorem}

Note that if the pairs $(X,\Delta)$ and $(Y,\Gamma)$ both have $\mathbb{Q}$-factorial
terminal singularities then the birational map $\sigma$ is automatically an isomorphism in
codimension one.

We recall the definition of a Sarkisov link.  Suppose that $\phi\colon\map X.S.$ and
$\psi\colon\map Y.T.$ are two Mori fibre spaces.  A Sarkisov link $\sigma\colon\rmap X.Y.$
between $\phi$ and $\psi$ is one of four types:
\begin{align*}
&\begin{diagram}
          &    \text{I}      &    \\
  X'      & \rDashto  &  Y   \\
 \dTo      &           &  \dTo_{\psi}\\
X      &           &  T  \\
\dTo^{\phi} & \ldTo &        \\
S          &           & 
\end{diagram}
&&
\begin{diagram}
          &    \text{II}      &    \\
 X'      &     \rDashto &  Y' \\
 \dTo     &              &  \dTo     \\
 X        &              &   Y   \\
\dTo^{\phi}&             & \dTo_{\psi} \\
 S        &         =   &  T
\end{diagram}
&&
\begin{diagram}
          &    \text{III}      &    \\
 X          & \rDashto      &  Y'\\
 \dTo^{\phi}  &         &  \dTo   \\
 S    &                &  Y \\
       & \rdTo &   \dTo_{\psi}    \\
       &            &    T 
\end{diagram}
&&
\begin{diagram}
       &   &    \text{IV} & &   \\
X      &   & \rDashto &    &  Y   \\
\dTo^{\phi} &       &      &   &  \dTo_{\psi}      \\
 S         &        &      &  &  T  \\
           & \rdTo  &      & \ldTo  &        \\
           &        &  R.   &        & 
\end{diagram}
\end{align*} 
There is a divisor $\Xi$ on the space $L$ on the top left (be it $L=X$ or $L=X'$) such
that $K_L+\Xi$ is kawamata log terminal and numerically trivial over the base (be it $S$,
$T$, or $R$).  Every arrow which is not horizontal is an extremal contraction.  If the
target is $X$ or $Y$ it is a divisorial contraction.  The horizontal dotted arrows are
compositions of $(K_L+\Xi)$-flops.  Links of type IV break into two types, IV${}_m$ and
IV${}_s$.  For a link of type IV${}_m$ both $s$ and $t$ are Mori fibre spaces.  For a link
of type IV${}_s$ both $s$ and $t$ are small birational contractions.  In this case $R$ is
not $\mathbb{Q}$-factorial; for every other type of link all varieties are
$\mathbb{Q}$-factorial.  Note that there is an induced birational map $\sigma\colon\rmap
X.Y.$ but not necessarily a rational map between $S$ and $T$.

The Sarkisov program has its origin in the birational classification of ruled surfaces.  
A link of type I corresponds to the diagram
\begin{diagram}
\Hz 1.      & =  &  \Hz 1.   \\
 \dTo      &           &  \dTo_{\psi}\\
\pr 2.      &           &  \pr 1.  \\
\dTo^{\phi} & \ldTo &        \\
\text{pt.}   &           & 
\end{diagram}
Note that there are no flops for surfaces so the top horizontal map is always the
identity.  The top vertical arrow on the left is the blow up of a point in $\pr 2.$ and
$\psi$ is the natural map given by the pencil of lines.  A link of type III is the same
diagram, reflected in a vertical line,
\begin{diagram}
\Hz 1.      & =  &  \Hz 1.   \\
\dTo^{\phi} &           &  \dTo \\
\pr 1.      &           &  \pr 2.  \\
 & \rdTo &    \dTo_{\psi}    \\
          &           & \text{pt.}
\end{diagram}

A link of type II corresponds to the classical elementary transformation between ruled
surfaces,
\begin{diagram}
X'    & =  & Y'    \\
\dTo    &  & \dTo    \\
X    & &  Y   \\
\dTo^{\phi}  & &    \dTo_{\psi}    \\
S      &    =       &  T.  \\
\end{diagram}
The birational map $\map X'.X.$ blows up a point in one fibre and the birational map $\map
Y'.Y.$ blows down the old fibre.  Finally a link of type IV corresponds to switching
between the two ways to project $\pr 1.\times \pr 1.$ down to $\pr 1.$,
$$
\begin{diagram}
\pr 1.\times\pr 1.   &   &= &  & \pr 1.\times \pr 1.   \\
\dTo^{\phi} & & & & \dTo_{\psi} \\
\pr 1.  &   & &  & \pr 1. \\
& \rdTo   &      & \ldTo & \\
     &   & \text{pt.}     & &
\end{diagram}
$$
It is a fun exercise to factor the classical Cremona transformation $\sigma\colon\rmap
{\pr 2.}.{\pr 2.}.$, $\map [X:Y:Z].[X^{-1}:Y^{-1}:Z^{-1}].$ into a product of Sarkisov
links.  Indeed one can use the Sarkisov program to give a very clean proof that the
birational automorphism of $\pr 2.$ is generated by this birational map $\sigma$ and
$\pgl(3)$.  More generally the Sarkisov program can sometimes be used to calculate the
birational automorphism group of Mori fibre spaces, especially Fano varieties.  With this
said, note that the following problem seems quite hard:
\begin{question}\label{q_three} What are generators of the birational automorphism group
of $\pr 3.$?
\end{question}

Note that a link of type IV${}_s$ only occurs in dimension four or more.  For an example
of a link of type IV${}_s$ simply take $\rmap S.T.$ to be a flop between threefolds, let
$\map S.R.$ be the base of the flop and let $X=S\times \pr 1.$ and $Y=T\times \pr 1.$ with
the obvious maps down to $S$ and $T$.  It is conceivable that one can factor a link of
type IV${}_s$ into links of type I and III.  However given any positive integer $k$ it is
easy to write down examples of links of type IV which cannot be factored into fewer than
$k$ links of type I, II or III.

Let us now turn to a description of the proof of \eqref{t_sarkisov}.  The proof is based
on the original ideas of the Sarkisov program (as explained by Corti and Reid
\cite{Corti95}; see also \cite{BM97a}).  We are given a birational map $\sigma\colon\rmap
X.Y.$ and the objective is to factor $\sigma$ into a product of Sarkisov links.  In the
original proof one keeps track of some subtle invariants and the idea is to prove:
\begin{itemize}
\item the first Sarkisov link $\sigma_1$ exists,  
\item if one chooses $\sigma_1$ appropriately then the invariants improve, and 
\item the invariants cannot increase infinitely often.
\end{itemize} 

Sarkisov links arise naturally if one plays the $2$-ray game.  If the relative Picard
number is two then there are only two rays to contract and this gives a natural way to
order the steps of the minimal model program.  One interesting feature of the original
proof is that it is a little tricky to prove the existence of the first Sarkisov link,
even if we assume existence and termination of flips.  In the original proof one picks a
linear system on $Y$ and pulls it back to $X$.  There are then three invariants to keep
track of; the singularities of the linear system on $X$, as measured by the canonical
threshold, the number of divisors of log discrepancy one (after rescaling to the canonical
threshold) and the pseudo-effective threshold.  Even for threefolds it is very hard to
establish that these invariants satisfy the ascending chain condition.

Our approach is quite different.  We don't consider any linear systems nor do we try to
keep track of any invariants.  Instead we use one of the main results of \cite{BCHM06},
namely finiteness of ample models for kawamata log terminal pairs $(Z,A+B)$.  Here $A$ is
a fixed ample $\mathbb{Q}$-divisor and $B$ ranges over a finite dimensional affine space
of Weil divisors.  The closure of the set of divisors $B$ with the same ample model is a
disjoint union of finitely many polytopes and the union of all of these polytopes
corresponds to divisors in the effective cone.

Now if the space of Weil divisors spans the N\'eron-Severi group then one can read off
which ample model admits a contraction to another ample model from the combinatorics of
the polytopes, \eqref{t_polytope}.  Further this property is preserved on taking a general
two dimensional slice, \eqref{c_polytope}.  Sarkisov links then correspond to points on
the boundary of the effective cone which are contained in more than two polytopes,
\eqref{t_two}.  To obtain the required factorisation it suffices to simply traverse the
boundary.  In other words instead of considering the closed cone of curves and playing the
$2$-ray game we look at the dual picture of Weil divisors and we work inside a carefully
chosen two dimensional affine space.  The details of the correct choice of underlying
affine space are contained in \S 4.

To illustrate some of these ideas, let us consider an easy case.  Let $S$ be the blow up
of $\pr 2.$ at two points.  Then $S$ is a toric surface and there are five invariant
divisors.  The two exceptional divisors, $E_1$ and $E_2$, the strict transform $L$ of the
line which meets $E_1$ and $E_2$, and finally the strict transform $L_1$ and $L_2$ of two
lines, one of which meets $E_1$ and one of which meets $E_2$.  Then the cone of effective
divisors is spanned by the invariant divisors and according to \cite{HK00} the polytopes
we are looking for are obtained by considering the chamber decomposition given by the
invariant divisors.  Since $L_1=L+E_1$ and $L_2=L+E_2$ the cone of effective divisors is
spanned by $L$, $E_1$ and $E_2$.  Since $-K_S$ is ample, we can pick an ample
$\mathbb{Q}$-divisor $A$ such that $K_S+A \sim_{\mathbb{Q}} 0$ and $K_S+A+E_1+E_2+L$ is
divisorially log terminal.  Let $V$ be the real vector space of Weil divisors spanned by
$E_1$, $E_2$ and $L$.  In this case projecting $\mathcal{L}_A(V)$ from the origin we get

\includegraphics[bbllx=146,bblly=500,bburx=415,bbury=662]{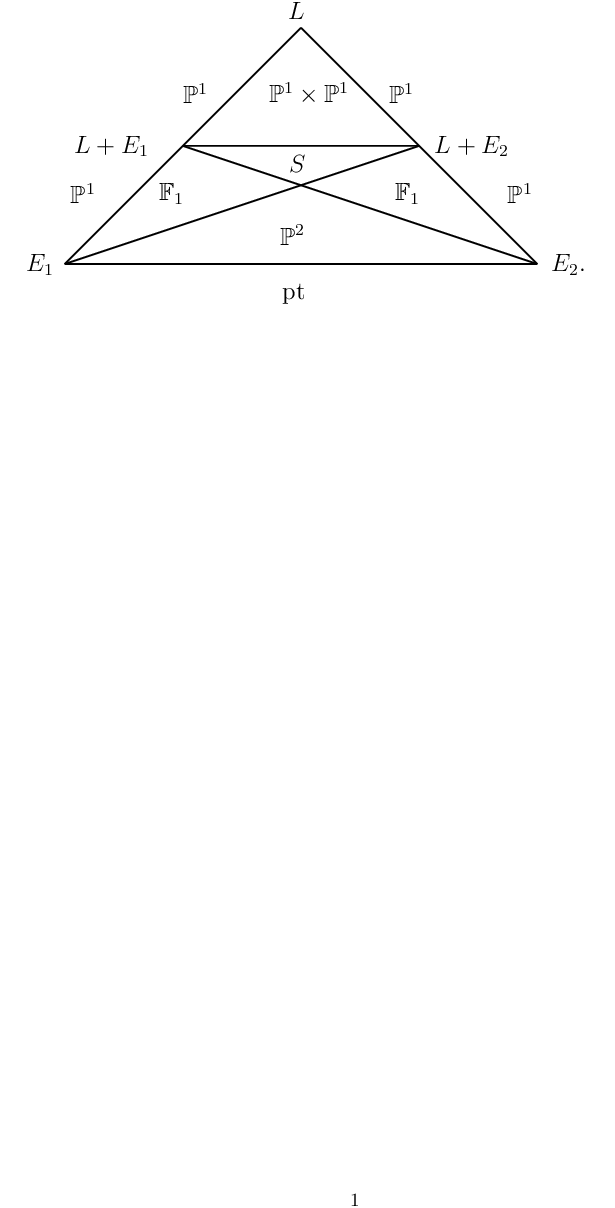}

We have labelled each polytope by the corresponding model.  Imagine going around the
boundary clockwise, starting just before the point corresponding to $L$.  The point $L$
corresponds to a Sarkisov link of type IV${}_m$, the point $L+E_2$ a link of type II, the
point $E_2$ a link of type III, the point $E_1$ a link of type I and the point $L+E_1$
another link of type II.

\section{Notation and conventions}
\label{s-notation}

We work over the field of complex numbers $\mathbb{C}$.  An $\mathbb{R}$-Cartier divisor
$D$ on a variety $X$ is \textit{nef} if $D\cdot C\geq 0$ for any curve $C\subset X$.  We
say that two $\mathbb{R}$-divisors $D_1$, $D_2$ are $\mathbb{R}$-linearly equivalent
($D_1\sim _{\mathbb{R}} D_2$) if $D_1-D_2=\sum r_i(f_i)$ where $r_i\in \mathbb{R}$ and
$f_i$ are rational functions on $X$.  We say that an $\mathbb{R}$-Weil divisor $D$ is
\textit{big} if we may find an ample $\mathbb{R}$-divisor $A$ and an $\mathbb{R}$-divisor
$B\geq 0$, such that $D \sim _{\mathbb{R}} A+B$.  A divisor $D$ is
\textit{pseudo-effective}, if for any ample divisor $A$ and any rational number $\epsilon
>0$, the divisor $D+\epsilon A$ is big.  If $A$ is a $\mathbb Q$-divisor, we say that $A$
is a \textit{general ample $\mathbb{Q}$-divisor} if $A$ is ample and there is a
sufficiently divisible integer $m>0$ such that $mA$ is very ample and $mA\in |mA|$ is very
general.

A \textit{log pair} $(X,\Delta)$ is a normal variety $X$ and an $\mathbb{R}$-Weil divisor
$\Delta\geq 0$ such that $K_X+\Delta$ is $\mathbb{R}$-Cartier.  We say that a log pair
$(X,\Delta)$ is \textit{log smooth}, if $X$ is smooth and the support of $\Delta$ is a
divisor with global normal crossings.  A projective birational morphism $g\colon \map
Y.X.$ is a \textit{log resolution} of the pair $(X,\Delta )$ if $Y$ is smooth and the
strict transform $\Gamma$ of $\Delta$ union the exceptional set $E$ of $g$ is a divisor
with normal crossings support.  If we write
$$
K_Y+\Gamma+E=g^*(K_X +\Delta)+\sum a_iE_i,
$$
where $E=\sum E_i$ is the sum of the exceptional divisors then the log discrepancy
$a(E_i,X,\Delta)$ of $E_i$ is $a_i$.  By convention the log discrepancy of any divisor $B$
which is not exceptional is $1-b$, where $b$ is the coefficient of $B$ in $\Delta$.  The
log discrepancy $a$ is the infinimum of the log discrepancy of any divisor.  

A pair $(X,\Delta)$ is \textit{kawamata log terminal} if $a>0$.  We say that the pair
$(X,\Delta)$ is \textit{log canonical} if $a\geq 0$.  We say that the pair $(X,\Delta)$ is
\textit{terminal} if the log discrepancy of any exceptional divisor is greater than one.

We say that a rational map $\phi\colon\rmap X.Y.$ is a \textit{rational contraction} if
there is a resolution $p\colon\map W.X.$ and $q\colon\map W.Y.$ of $\phi$ such that $p$
and $q$ are contraction morphisms and $p$ is birational.  We say that $\phi$ is a
\textit{birational contraction} if $q$ is in addition birational and every $p$-exceptional
divisor is $q$-exceptional.  If in addition $\phi^{-1}$ is also a birational contraction,
we say that $\phi$ is a \textit{small birational map}.  We refer the reader to
\cite{BCHM06} for the definitions of negative and non-positive rational contractions
and of log terminal models.  

If $\mathcal{C}$ is a closed convex in a finite dimensional real vector space then 
$\mathcal{C}^*$ denotes the dual convex set in the dual real vector space.  

\section{The combinatorics of ample models}

We fix some notation.  $Z$ is a smooth projective variety, $V$ is a finite dimensional
affine subspace of the real vector space $\WDiv_{\mathbb{R}}(Z)$ of Weil divisors on $Z$,
which is defined over the rationals, and $A\geq 0$ is an ample $\mathbb{Q}$-divisor on
$Z$.  We suppose that there is an element $\Theta_0$ of $\mathcal{L}_A(V)$ such that
$K_Z+\Theta_0$ is big and kawamata log terminal.

We recall some definitions and notation from \cite{BCHM06}:
\begin{definition}\label{d_ample} Let $D$ be an $\mathbb{R}$-divisor on $Z$.

We say that $f\colon\rmap Z.X.$ is the \textbf{ample model} of $D$, if $f$ is a rational
contraction, $X$ is a normal projective variety and there is an ample divisor $H$ on $X$
such that if $p\colon\map W.Z.$ and $q\colon\map W.X.$ resolve $f$ and we write
$p^*D\sim_{\mathbb{R}}q^*H+E$, then $E\geq 0$ and for every $B\sim_{\mathbb{R}} p^*D$ if
$B\geq 0$ then $B\geq E$.
\end{definition}

Note that if $f$ is birational then $q_*E=0$.  

\begin{definition}\label{d_cones} Let 
\begin{align*} 
V_A &=\{\,\Theta \,|\, \Theta=A+B, B\in V \,\}, \\
\mathcal{L}_A(V)&=\{\,\Theta=A+B\in V_A \,|\, \text{$K_Z+\Theta$ is log canonical and $B\geq 0$} \,\}, \\
\mathcal{E}_A(V) &=\{\,\Theta\in \mathcal{L}_A(V) \,|\, \text{$K_Z+\Theta$ is pseudo-effective} \,\}.
\end{align*}
Given a rational contraction $f\colon\rmap Z.X.$, define
$$
\mathcal{A}_{A,f}(V)=\{\, \Theta\in \mathcal E_A(V) \,|\, \text{$f$ is the ample model of $(Z,\Theta)$} \,\}.
$$
\end{definition}

In addition, let $\mathcal{C}_{A,f}(V)$ denote the closure of $\mathcal{A}_{A,f}(V)$.  

\begin{theorem}\label{t_polytope} There are finitely many $1\leq i\leq m$ rational 
contractions $f_i\colon\rmap Z.X_i.$ with the following properties:
\begin{enumerate} 
\item $\ds{\{\, \mathcal{A}_i=\mathcal{A}_{A,f_i}\,|\, 1\leq i\leq m \,\}}$ is a partition
of $\mathcal{E}_{A}(V)$.  $\mathcal{A}_i$ is a finite union of interiors of rational
polytopes.  If $f_i$ is birational then $\mathcal{C}_i=\mathcal{C}_{A,f_i}$ is a rational
polytope.
\item If $1\leq i\leq m$ and $1\leq j\leq m$ are two indices such that
$\mathcal{A}_j\cap\mathcal{C}_i\neq\varnothing$ then there is a contraction morphism
$f_{i,j}\colon\map X_i.X_j.$ and a factorisation $f_j=f_{i,j}\circ f_i$.
\end{enumerate}
Now suppose in addition that $V$ spans the N\'eron-Severi group of $Z$.
\begin{enumerate} 
\setcounter{enumi}{2}
\item Pick $1\leq i\leq m$ such that a connected component $\mathcal{C}$ of
$\mathcal{C}_i$ intersects the interior of $\mathcal{L}_A(V)$.  The following are
equivalent
\begin{itemize} 
\item $\mathcal{C}$ spans $V$.
\item If $\Theta\in \mathcal{A}_i\cap \mathcal{C}$ then $f_i$ is a log terminal model of
$K_Z+\Theta$.
\item $f_i$ is birational and $X_i$ is $\mathbb{Q}$-factorial.
\end{itemize}
\item If $1\leq i\leq m$ and $1\leq j \leq m$ are two indices such that $\mathcal{C}_i$
spans $V$ and $\Theta$ is a general point of $\mathcal{A}_j\cap \mathcal{C}_i$ which is
also a point of the interior of $\mathcal{L}_A(V)$ then $\mathcal{C}_i$ and $\ccone
X_i/X_j.^*\times \mathbb{R}^k$ are locally isomorphic in a neighbourhood of $\Theta$, for
some $k\geq 0$.  Further the relative Picard number of $f_{i,j}\colon\map X_i.X_j.$ is
equal to the difference in the dimensions of $\mathcal{C}_i$ and $\mathcal{C}_j\cap
\mathcal{C}_i$.
\end{enumerate} 
\end{theorem}
\begin{proof} (1) is proved in \cite{BCHM06}.  

Pick $\Theta\in\mathcal{A}_j\cap\mathcal{C}_i$ and $\Theta'\in\mathcal{A}_i$ so that 
$$
\Theta_t=\Theta+t(\Theta'-\Theta)\in \mathcal{A}_i \qquad \text{if} \qquad t\in (0,1].
$$
By finiteness of log terminal models, cf. \cite{BCHM06}, we may find a positive constant
$\delta>0$ and a birational contraction $f\colon\rmap Z.X.$ which is a log terminal model
of $K_Z+\Theta_t$ for $t\in (0,\delta]$.  Replacing $\Theta'=\Theta_1$ by
$\Theta_{\delta}$ we may assume that $\delta=1$.  If we set
$$
\Delta_t=f_*\Theta_t,
$$
then $K_X+\Delta_t$ is kawamata log terminal and nef, and $f$ is $K_Z+\Theta_t$
non-positive for $t\in [0,1]$.  As $\Delta_t$ is big the base point free theorem implies
that $K_X+\Delta_t$ is semiample and so there is an induced contraction morphism
$g_i\colon\map X.X_i.$ together with ample divisors $H_{1/2}$ and $H_1$ such that
$$
K_X+\Delta_{1/2}=g_i^*H_{1/2} \qquad \text{and} \qquad K_X+\Delta_1=g_i^*H_1.
$$
If we set 
$$
H_t=(2t-1)H_1+2(1-t)H_{1/2},
$$
then 
\begin{align*} 
K_X+\Delta_t &= (2t-1)(K_X+\Delta_1)+2(1-t)(K_X+\Delta_{1/2}) \\
             &= (2t-1)g_i^*H_1+2(1-t)g_i^*H_{1/2} \\
             &= g_i^*H_t,
\end{align*} 
for all $t\in [0,1]$.  As $K_X+\Delta_0$ is semiample, it follows that $H_0$ is semiample
and the associated contraction $f_{i,j}\colon\map X_i.X_j.$ is the required morphism.  
This is (2).  

Now suppose that $V$ spans the N\'eron-Severi group of $Z$.  Suppose that $\mathcal{C}$
spans $V$.  Pick $\Theta$ in the interior of $\mathcal{C}\cap \mathcal{A}_i$.  Let
$f\colon\rmap Z.X.$ be a log terminal model of $K_Z+\Theta$.  It is proved in
\cite{BCHM06} that $f=f_j$ for some index $1\leq j\leq m$ and that $\Theta\in
\mathcal{C}_j$.  But then $\mathcal{A}_i\cap \mathcal{A}_j\neq\varnothing$ so 
that $i=j$.  

If $f_i$ is a log terminal model of $K_Z+\Theta$ then $f_i$ is birational and $X_i$ is
$\mathbb{Q}$-factorial.

Finally suppose that $f_i$ is birational and $X_i$ is $\mathbb{Q}$-factorial.  Fix
$\Theta\in \mathcal{A}_i$.  Pick any divisor $B\in V$ such that $-B$ is ample
$K_{X_i}+f_{i*}(\Theta+B)$ is ample and $\Theta+B\in \mathcal{L}_A(V)$.  Then $f_i$ is
$(K_Z+\Theta+B)$-negative and so $\Theta+B\in \mathcal{A}_i$.  But then $\mathcal{C}_i$
spans $V$.  This is (3).

We now prove (4).  Let $f=f_i$ and $X=X_i$.  As $\mathcal{C}_i$ spans $V$, (3) implies
that $f$ is birational and $X$ is $\mathbb{Q}$-factorial so that $f$ is a
$\mathbb{Q}$-factorial weak log canonical model of $K_Z+\Theta$.  Suppose that $\llist
E.k.$ are the divisors contracted by $f$.  Pick $B_i\in V$ numerically equivalent to
$E_i$.  If we let $E_0=\sum E_i$ and $B_0=\sum B_i$ then $E_0$ and $B_0$ are numerically
equivalent.  As $\Theta$ belongs to the interior of $\mathcal{L}_A(V)$ we may find
$\delta>0$ such that $K_Z+\Theta+\delta E_0$ and $K_Z+\Theta+\delta B_0$ are both kawamata
log terminal.  Then $f$ is $(K_Z+\Theta+\delta E_0)$-negative and so $f$ is a log terminal
model of $K_Z+\Theta+\delta E_0$ and $f_j$ is the ample model of $K_Z+\Theta+\delta E_0$.
But then $f$ is also a log terminal model of $K_Z+\Theta+\delta B_0$ and $f_j$ is also the
ample model of $K_Z+\Theta+\delta B_0$.  In particular $\Theta+\delta B_0\in
\mathcal{A}_j\cap \mathcal{C}_i$.  As we are supposing that $\Theta$ is general in
$\mathcal{A}_j\cap \mathcal{C}_i$, in fact $f$ must be a log terminal model of
$K_Z+\Theta$.  In particular $f$ is $(K_Z+\Theta)$-negative.

Pick $\epsilon>0$ such that if $\Xi\in V$ and $\|\Xi-\Theta\|<\epsilon$ then $\Xi$ belongs
to the interior of $\mathcal{L}_A(V)$ and $f$ is $(K_Z+\Xi)$-negative.  Then the condition
that $\Xi\in \mathcal{C}_i$ is simply the condition that $K_X+\Delta=f_*(K_Z+\Xi)$ is nef.
Let $W$ be the affine suspace of $\WDiv_{\mathbb{R}}(X)$ given by pushing forward the
elements of $V$ and let
$$
\mathcal{N}=\{\, \Delta\in W \,|\, \text{$K_X+\Delta$ is nef} \,\}.  
$$
Given $(\llist a.k.)\in \mathbb{R}^k$ let $B=\sum a_iB_i$ and $E=\sum a_iE_i$.  If
$\|B\|<\epsilon$ then, as $\Xi+B$ is numerically equivalent to $\Xi+E$, $K_X+\Delta\in
\mathcal{N}$ if and only if $K_X+\Delta+f_*B\in \mathcal{N}$.  In particular
$\mathcal{C}_i$ is locally isomorphic to $\mathcal{N}\times \mathbb{R}^k$.

But since $f_j$ is the ample model of $K_Z+\Theta$, in fact we can choose $\epsilon$
sufficiently small so that $K_X+\Delta$ is nef if and only if $K_X+\Delta$ is nef over
$X_j$, see \S 3 of \cite{BCHM06}.  There is a surjective affine linear map from 
$W$ to the space of Weil divisors on $X$ modulo numerical equivalence over $X_j$ 
and this induces an isomorphism 
$$
\mathcal{N}\simeq \ccone X/X_j.^*\times \mathbb{R}^l,
$$
in a neighbourhood of $f_*\Theta$.

Note that $K_X+f_*\Theta$ is numerically trivial over $X_j$.  As $f_*\Theta$ is big and
$K_X+f_*\Theta$ is kawamata log terminal we may find an ample $\mathbb{Q}$-divisor $A'$
and a divisor $B'\geq 0$ such that
$$
K_X+A'+B' \sim_{\mathbb{R}} K_X+f_*\Theta,
$$
is kawamata log terminal.  But then
$$
-(K_X+B') \sim_{\mathbb{R}} -(K_X+\Delta')+A',
$$ 
is ample over $X_j$.  Hence $f_{ij}\colon\map X.X_j.$ is a Fano fibration and 
so by the cone theorem 
$$
\rho(X_i/X_j)=\dim \mathcal{N}.
$$
This is (4).  \end{proof}

\begin{corollary}\label{c_polytope} If $V$ spans the N\'eron-Severi group of $Z$ then
there is a Zariski dense open subset $U$ of the Grassmannian $G(\alpha,V)$ of real affine
subspaces of dimension $\alpha$ such that if $[W]\in U$ and it is defined over the
rationals then $W$ satisfies (1-4) of \eqref{t_polytope}.
\end{corollary}
\begin{proof} Let $U\subset G(\alpha,V)$ be the set of real affine subspaces $W$ of $V$ of
dimension $\alpha$, which contain no face of any $\mathcal{C}_i$ or $\mathcal{L}_A(V)$.
In particular the interior of $\mathcal{L}_A(W)$ is contained in the interior of
$\mathcal{L}_A(V)$.  \eqref{t_polytope} implies that (1-2) always hold for $W$ and (1-4)
hold for $V$ and so (3) and (4) clearly hold for $W\in U$.  \end{proof}

From now on in this section we assume that $V$ has dimension two and satisfies (1-4) of
\eqref{t_polytope}.

\begin{lemma}\label{l_easy-cases} Let $f\colon\rmap Z.X.$ and $g\colon\rmap Z.Y.$ 
be two rational contractions such that $\mathcal{C}_{A,f}$ is two dimensional and
$\mathcal{O}=\mathcal{C}_{A,f}\cap \mathcal{C}_{A,g}$ is one dimensional.  Assume that
$\rho(X)\geq \rho(Y)$ and that $\mathcal{O}$ is not contained in the boundary of
$\mathcal{L}_A(V)$.  Let $\Theta$ be an interior point of $\mathcal{O}$ and let
$\Delta=f_*\Theta$.

Then there is a rational contraction $\pi\colon\rmap X.Y.$ which factors $g=\pi\circ f$
and either
\begin{enumerate} 
\item $\rho(X)=\rho(Y)+1$ and $\pi$ is a $(K_X+\Delta)$-trivial morphism, in which case,
either
\begin{enumerate}
\item $\pi$ is birational and $\mathcal{O}$ is not contained in the boundary of
$\mathcal{E}_A(V)$, in which case, either
\begin{enumerate}
\item $\pi$ is a divisorial contraction and $\mathcal{O}\neq\mathcal{C}_{A,g}$, or
\item $\pi$ is a small contraction and $\mathcal{O}=\mathcal{C}_{A,g}$, or
\end{enumerate}
\item $\pi$ is a Mori fibre space and $\mathcal{O}=\mathcal{C}_{A,g}$ is contained in the
boundary of $\mathcal{E}_A(V)$, or
\end{enumerate} 
\item $\rho(X)=\rho(Y)$, in which case, $\pi$ is a $(K_X+\Delta)$-flop and
$\mathcal{O}\neq \mathcal{C}_{A,g}$ is not contained in the boundary of
$\mathcal{E}_A(V)$.
\end{enumerate} 
\end{lemma}
\begin{proof} By assumption $f$ is birational and $X$ is $\mathbb{Q}$-factorial.  Let
$h\colon\rmap Z.W.$ be the ample model corresponding to $K_Z+\Theta$.  Since $\Theta$ is
not a point of the boundary of $\mathcal{L}_A(V)$ if $\Theta$ belongs to the boundary of
$\mathcal{E}_A(V)$ then $K_Z+\Theta$ is not big and so $h$ is not birational.  As
$\mathcal{O}$ is a subset of both $\mathcal{C}_{A,f}$ and $\mathcal{C}_{A,g}$ there are
morphisms $p\colon\map X.W.$ and $q\colon\map Y.W.$ of relative Picard number at most one.
There are therefore only two possibilities:
\begin{enumerate} 
\item $\rho(X)=\rho(Y)+1$, or 
\item $\rho(X)=\rho(Y)$.
\end{enumerate} 

Suppose we are in case (1).  Then $q$ is the identity and $\pi=p\colon\map X.Y.$ is a
contraction morphism such that $g=\pi\circ f$.  Suppose that $\pi$ is birational.  Then
$h$ is birational and $\mathcal{O}$ is not contained in the boundary of
$\mathcal{E}_A(V)$.  If $\pi$ is divisorial then $Y$ is $\mathbb{Q}$-factorial and so
$\mathcal{O}\neq\mathcal{C}_{A,g}$.  If $\pi$ is a small contraction then $Y$ is not
$\mathbb{Q}$-factorial and so $\mathcal{C}_{A,g}=\mathcal{O}$ is one dimensional.  If
$\pi$ is a Mori fibre space then $\mathcal{O}$ is contained in the boundary of
$\mathcal{E}_A(V)$ and $\mathcal{O}=\mathcal{C}_{A,g}$.

Now suppose we are in case (2).  By what we have already proved $\rho(X/W)=\rho(Y/W)=1$.
$p$ and $q$ are not divisorial contractions as $\mathcal{O}$ is one dimensional.  $p$ and
$q$ are not Mori fibre spaces as $\mathcal{O}$ cannot be contained in the boundary of
$\mathcal{E}_A(V)$.  Hence $p$ and $q$ are small and the rest is clear.  \end{proof}

\begin{lemma}\label{l_negative} Let $f\colon\rmap W.X.$ be a birational contraction between 
projective $\mathbb{Q}$-factorial varieties.  Suppose that $(W,\Theta)$ and $(W,\Phi)$
are both kawamata log terminal.  

If $f$ is the ample model of $K_W+\Theta$ and $\Theta-\Phi$ is ample then $f$ is the
result of running the $(K_W+\Phi)$-MMP.
\end{lemma}
\begin{proof} By assumption we may find an ample divisor $H$ on $W$ such that $K_W+\Phi+H$
is kawamata log terminal and ample and a positive real number $t<1$ such that $tH
\sim_{\mathbb{R}} \Theta-\Phi$.  Note that $f$ is the ample model of $K_W+\Phi+tH$.  Pick
any $s<t$ sufficiently close to $t$ so that $f$ is $(K_W+\Phi+sH)$-negative and yet $f$ is
still the ample model of $K_W+\Phi+sH$.  Then $f$ is the unique log terminal model of
$K_W+\Phi+sH$.  In particular if we run the $(K_W+\Phi)$-MMP with scaling of $H$ then,
when the value of the scalar is $s$, the induced rational map is $f$.  \end{proof}

We now adopt some more notation for the rest of this section.  Let $\Theta=A+B$ be a point
of the boundary of $\mathcal{E}_A(V)$ in the interior of $\mathcal{L}_A(V)$.  Enumerate
$\llist \mathcal{T}.k.$ the polytopes $\mathcal{C}_i$ of dimension two which contain
$\Theta$.  Possibly re-ordering we may assume that the intersections $\mathcal{O}_0$ and
$\mathcal{O}_k$ of $\mathcal{T}_1$ and $\mathcal{T}_k$ with the boundary of
$\mathcal{E}_A(V)$ and $\mathcal{O}_i=\mathcal{T}_i\cap \mathcal{T}_{i+1}$ are all one
dimensional.  Let $f_i\colon\rmap Z.X_i.$ be the rational contractions associated to
$\mathcal{T}_i$ and $g_i\colon\rmap Z.S_i.$ be the rational contractions associated to
$\mathcal{O}_i$.  Set $f=f_1\colon\rmap Z.X.$, $g=f_k\colon\rmap Z.Y.$, $X'=X_2$,
$Y'=X_{k-1}$.  Let $\phi\colon\map X.S=S_0.$, $\psi\colon\map Y.T=S_k.$ be the induced
morphisms and let $\rmap Z.R.$ be the ample model of $K_Z+\Theta$.

\includegraphics[bbllx=90,bblly=550,bburx=487,bbury=666]{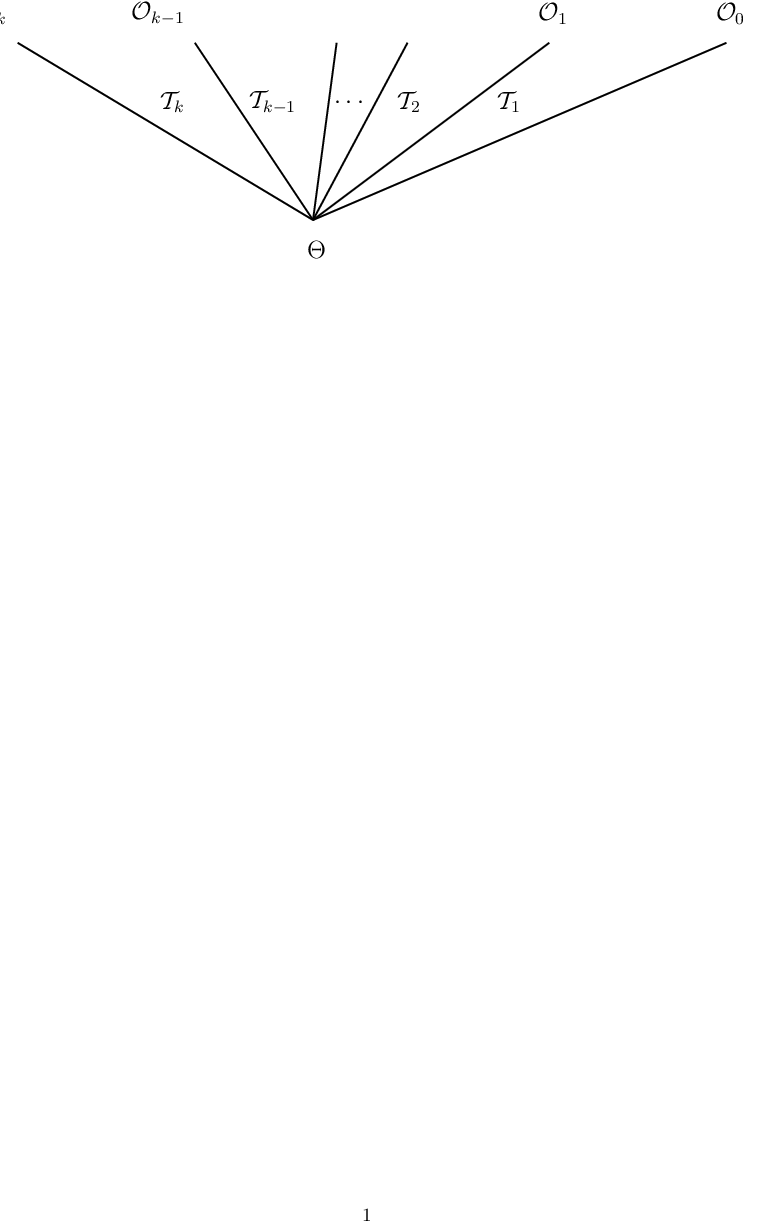}

\begin{theorem}\label{t_two} Suppose $\Phi$ is any divisor such that $K_Z+\Phi$ is
kawamata log terminal and $\Theta-\Phi$ is ample.

Then $\phi$ and $\psi$ are two Mori fibre spaces which are outputs of the $(K_Z+\Phi)$-MMP
which are connected by a Sarkisov link if $\Theta$ is contained in more than two
polytopes.  
\end{theorem}
\begin{proof} We assume for simplicity of notation that $k\geq 3$.  The case $k\leq 2$ is
similar and we omit it.  The incidence relations between the corresponding polytopes yield
a commutative heptagon,
$$
\begin{diagram}
X'           &               & \rDashto &  &  Y'  \\
\dDashto^p           &   &      &            & \dDashto_q      \\
X          &               &      &            &  Y   \\
\dTo^{\phi} &              &      &            &  \dTo_{\psi}      \\
 S         &               &      &            &  T  \\
           & \rdTo(2,2)_s  &      & \ldTo(2,2)_t          &        \\
           &               &  R   &            & 
\end{diagram}
$$
where $p$ and $q$ are birational maps.  $\phi$ and $\psi$ are Mori fibre spaces by
\eqref{l_easy-cases}.  Pick $\Theta_1$ and $\Theta_k$ in the interior of $\mathcal{T}_1$
and $\mathcal{T}_k$ sufficiently close to $\Theta$ so that $\Theta_1-\Phi$ and
$\Theta_k-\Phi$ are ample.  As $X$ and $Y$ are $\mathbb{Q}$-factorial, \eqref{l_negative}
implies that $\phi$ and $\psi$ are possible outcomes of the $(K_Z+\Phi)$-MMP.  Let
$\Delta=f_*\Theta$.  Then $K_X+\Delta$ is numerically trivial over $R$.

Note that there are contraction morphisms $\map X_i.R.$ and that $\rho(X_i/R)\leq 2$.  If
$\rho(X_i/R)=1$ then $\map X_i.R.$ is a Mori fibre space.  By \eqref{t_polytope} there is
facet of $\mathcal{T}_i$ which is contained in the boundary of $\mathcal{E}_A(V)$ and so
$i=1$ or $k$.  Thus $\rmap X_i.X_{i+1}.$ is a flop, $1<i<k-1$.  Since $\rho(X'/R)=2$ it
follows that either $p$ is a divisorial contraction and $s$ is the identity or $p$ is a
flop and $s$ is not the identity.  We have a similar dichotomy for $q\colon\rmap Y'.Y.$
and $t\colon\map T.R.$.

There are then four cases.  If $s$ and $t$ are the identity then $p$ and $q$ are 
divisorial extractions and we have a link of type II.  

If $s$ is the identity and $t$ is not then $p$ is a divisorial extraction and $q$ is a
flop and we have a link of type I.  Similarly if $t$ is the identity and $s$ is not then
$q$ is a divisorial extraction and $p$ is a flop and we have a link of type III.  

Finally suppose neither $s$ nor $t$ is the identity.  Then both $p$ and $q$ are flops.
Suppose that $s$ is a divisorial contraction.  Let $F$ be the divisor contracted by $s$
and let $E$ be its inverse image in $X$.  Since $\phi$ has relative Picard number one
$\phi^*(F)=mE$, for some positive integer $m$.  Then $K_X+\Delta+\delta E$ is kawamata log
terminal for any $\delta>0$ sufficiently small and $E=\mathbf{B}(K_X+\Delta+\delta E/R)$.
If we run the $(K_X+\Delta+\delta E)$-MMP over $R$ then we end with a birational
contraction $\rmap X.W.$, which is a Mori fibre space over $R$.  Since $\rho(X/R)=2$,
$W=Y$ and we have a link of type III, a contradiction.  Similarly $t$ is never a
divisorial contraction.  If $s$ is a Mori fibre space then $R$ is $\mathbb{Q}$-factorial
and so $t$ must be a Mori fibre space as well.  This is a link of type IV${}_m$.  If $s$
is small then $R$ is not $\mathbb{Q}$-factorial and so $t$ is small as well.  Thus we have
a link of type IV${}_s$.  \end{proof}

\section{Proof of \eqref{t_sarkisov} }

\begin{lemma}\label{l_perturb} Let $\phi\colon\map X.S.$ and $\psi\colon\map Y.T.$ be two 
Sarkisov related Mori fibre spaces corresponding to two $\mathbb{Q}$-factorial kawamata
log terminal projective varieties $(X,\Delta)$ and $(Y,\Gamma)$.

Then we may find a smooth projective variety $Z$, two birational contractions
$f\colon\rmap Z.X.$ and $g\colon\rmap Z.Y.$, a kawamata log terminal pair $(Z,\Phi)$, an
ample $\mathbb{Q}$-divisor $A$ on $Z$ and a two dimensional rational affine subspace $V$
of $\WDiv_{\mathbb{R}}(Z)$ such that
\begin{enumerate} 
\item if $\Theta\in \mathcal{L}_A(V)$ then $\Theta-\Phi$ is ample,
\item $\mathcal{A}_{A,\phi\circ f}$ and $\mathcal{A}_{A,\psi\circ g}$ are not
contained in the boundary of $\mathcal{L}_A(V)$, 
\item $V$ satisfies (1-4) of \eqref{t_polytope}, 
\item $\mathcal{C}_{A,f}$ and $\mathcal{C}_{A,g}$ are two dimensional, and
\item $\mathcal{C}_{A,\phi\circ  f}$ and $\mathcal{C}_{A,\psi\circ g}$ are one
dimensional.   
\end{enumerate} 
\end{lemma}
\begin{proof} By assumption we may find a $\mathbb{Q}$-factorial kawamata log terminal
pair $(Z,\Phi)$ such that $f\colon\rmap Z.X.$ and $g\colon\rmap Z.Y.$ are both outcomes of
the $(K_Z+\Phi)$-MMP.

Let $p\colon\map W.Z.$ be any log resolution of $(Z,\Phi)$ which resolves the
indeterminancy of $f$ and $g$.  We may write
$$
K_W+\Psi=p^*(K_Z+\Phi)+E',
$$
where $E'\geq 0$ and $\Psi\geq 0$ have no common components, $E'$ is exceptional and
$p_*\Psi=\Phi$.  Pick $-E$ ample over $Z$ with support equal to the full exceptional locus
such that $K_W+\Psi+E$ is kawamata log terminal.  As $p$ is $(K_W+\Psi+E)$-negative,
$K_Z+\Phi$ is kawamata log terminal and $Z$ is $\mathbb{Q}$-factorial, the
$(K_W+\Psi+E)$-MMP over $Z$ terminates with the pair $(Z,\Phi)$ by \eqref{l_negative}.
Replacing $(Z,\Phi)$ with $(W,\Psi+E)$, we may assume that $(Z,\Phi)$ is log smooth and
$f$ and $g$ are morphisms.

Pick general ample $\mathbb{Q}$-divisors $A, \llist H.k.$ on $Z$ such that $\llist H.k.$
generate the N\'eron-Severi group of $Z$.  Let
$$
H=A+\alist H.+.k..  
$$
Pick sufficiently ample divisors $C$ on $S$ and $D$ on $T$ such that
$$
-(K_X+\Delta)+\phi^*C \qquad \text{and} \qquad -(K_Y+\Gamma)+\psi^*D,
$$
are both ample.  Pick a rational number $0<\delta<1$ such that
$$
-(K_X+\Delta+\delta f_*H)+\phi^*C \qquad \text{and} \qquad -(K_Y+\Gamma+\delta g_*H)+\psi^*D,
$$
are both ample and $K_Z+\Phi+\delta H$ is both $f$ and $g$-negative.  Replacing $H$ by
$\delta H$ we may assume that $\delta=1$.  Now pick a $\mathbb{Q}$-divisor $\Phi_0\leq
\Phi$ such that $A+(\Phi_0-\Phi)$,
$$
-(K_X+f_*\Phi_0+ f_*H)+\phi^*C \quad \text{and} \quad -(K_Y+g_*\Phi_0+ g_*H)+\psi^*D,
$$
are all ample and $K_Z+\Phi_0+ H$ is both $f$ and $g$-negative.

 Pick general ample $\mathbb{Q}$-divisors $F_1\geq 0$ and $G_1\geq 0$
$$
F_1\sim_{\mathbb{Q}} -(K_X+f_*\Phi_0+ f_*H)+\phi^*C \quad \text{and} \quad G_1 \sim_{\mathbb{Q}} -(K_Y+g_*\Phi_0+ g_*H)+\psi^*D.
$$
Then
$$
K_Z+\Phi_0+ H+F+G,
$$
is kawamata log terminal, where $F=f^*F_1$ and $G=g^*G_1$.  

Let $V_0$ be the affine subspace of $\WDiv_{\mathbb{R}}(Z)$ which is the translate by
$\Phi_0$ of the vector subspace spanned by $\llist H.k.,F,G$.  Suppose that $\Theta=A+B\in
\mathcal{L}_A(V_0)$.  Then
$$
\Theta-\Phi=(A+\Phi_0-\Phi)+(B-\Phi_0),
$$
is ample, as $B-\Phi_0$ is nef by definition of $V_0$.  Note that $\Phi_0+F+H\in
\mathcal{A}_{A,\phi\circ f}(V_0)$, $\Phi_0+G+H\in \mathcal{A}_{A,\psi\circ g}(V_0)$, and
$f$, respectively $g$, is a weak log canonical model of $K_Z+\Phi_0+F+H$, respectively
$K_Z+\Phi_0+G+H$.  \eqref{t_polytope} implies that $V_0$ satisfies (1-4) of
\eqref{t_polytope}.

Since $\llist H.k.$ generate the N\'eron-Severi group of $Z$ we may find constants $\llist
h.k.$ such that $G$ is numerically equivalent to $\sum h_iH_i$.  Then $\Phi_0+F+\delta
G+H-\delta(\sum h_iH_i)$ is numerically equivalent to $\Phi_0+F+H$ and if $\delta>0$ is
small enough $\Phi_0+F+\delta G+H-\sum \delta h_iH_i\in \mathcal{L}_A(V_0)$.  Thus
$\mathcal{A}_{A,\phi\circ f}(V_0)$ is not contained in the boundary of
$\mathcal{L}_A(V_0)$.  Similarly $\mathcal{A}_{A,\psi\circ g}(V_0)$ is not contained in
the boundary of $\mathcal{L}_A(V_0)$.  In particular $\mathcal{A}_{A,f}(V_0)$ and
$\mathcal{A}_{A,g}(V_0)$ span $V_0$ and $\mathcal{A}_{A,\phi\circ f}(V_0)$ and
$\mathcal{A}_{A,\psi\circ g}(V_0)$ span affine hyperplanes of $V_0$, since
$\rho(X/S)=\rho(Y/T)=1$.

Let $V_1$ be the translate by $\Phi_0$ of the two dimensional vector space spanned by
$F+H-A$ and $F+G-A$.  Let $V$ be a small general perturbation of $V_1$, which is defined
over the rationals.  Then (2) holds.  (1) holds, as it holds for any two dimensional
subspace of $V_0$, (3) holds by \eqref{c_polytope} and this implies that (4) and (5)
hold.  \end{proof}

\begin{proof}[Proof of \eqref{t_sarkisov}] Pick $(Z,\Phi)$, $A$ and $V$ given by
\eqref{l_perturb}.  Pick points $\Theta_0\in \mathcal{A}_{A,\phi\circ f}(V)$ and
$\Theta_1\in\mathcal{A}_{A,\psi\circ g}(V)$ belonging to the interior of
$\mathcal{L}_A(V)$.  As $V$ is two dimensional, removing $\Theta_0$ and $\Theta_1$ divides
the boundary of $\mathcal{E}_A(V)$ into two parts.  The part which consists entirely of
divisors which are not big is contained in the interior of $\mathcal{L}_A(V)$.  Consider
tracing this boundary from $\Theta_0$ to $\Theta_1$.  Then there are finitely many $2\leq
i\leq l$ points $\Theta_i$ which are contained in more than two polytopes
$\mathcal{C}_{A,f_i}(V)$.  \eqref{t_two} implies that for each such point there is a
Sarkisov link $\sigma_i\colon\rmap X_i.Y_i.$ and $\sigma$ is the composition of these
links. \end{proof}

\bibliographystyle{/home/mckernan/Jewel/Tex/hamsplain}
\bibliography{/home/mckernan/Jewel/Tex/math}


\end{document}